\documentclass[a4paper,12pt,reqno]{amsart}

\usepackage{amsmath, amsfonts, amssymb, amsthm, mathrsfs, mathtools}
\usepackage{comment}
\usepackage{enumerate}
\usepackage{enumitem}
\usepackage{fancyhdr}
\usepackage[pdfencoding=auto, hyperfootnotes=false, pagebackref]{hyperref}
\renewcommand*{\backref}[1]{}%
\renewcommand*{\backrefalt}[4]{[$\uparrow${\ifcase #1 Not cited.%
          \else \;#2%
          \fi%
    }]}

\usepackage{tikz,tikz-cd}
\usepackage[hang,flushmargin]{footmisc}
\usepackage{graphicx}
\usepackage{xr}
\usepackage[dvistyle, colorinlistoftodos]{todonotes}

\makeatletter
\def\@tocline#1#2#3#4#5#6#7{\relax
  \ifnum #1>\c@tocdepth 
  \else
    \par \addpenalty\@secpenalty\addvspace{#2}%
    \begingroup \hyphenpenalty\@M
    \@ifempty{#4}{%
      \@tempdima\csname r@tocindent\number#1\endcsname\relax
    }{%
      \@tempdima#4\relax
    }%
    \parindent\z@ \leftskip#3\relax \advance\leftskip\@tempdima\relax
    \rightskip\@pnumwidth plus4em \parfillskip-\@pnumwidth
    #5\leavevmode\hskip-\@tempdima
      \ifcase #1
       \or\or \hskip 1em \or \hskip 2em \else \hskip 3em \fi%
      #6\nobreak\relax
    \dotfill\hbox to\@pnumwidth{\@tocpagenum{#7}}\par
    \nobreak
    \endgroup
  \fi}
\makeatother

\newtheorem{theorem}{Theorem}[section]
\newtheorem{lemma}[theorem]{Lemma}

\newtheorem{corollary}[theorem]{Corollary}
\newtheorem{proposition}[theorem]{Proposition}
\newtheorem{question}[theorem]{Question}

\theoremstyle{definition}
\newtheorem{defn}[theorem]{Definition}
\newtheorem{remark}[theorem]{Remark}

\newcommand{\mc}{\mathcal}

\newcommand{\mb}{\mathbb}

\newcommand{\ud}{\,\mathrm{d}}

\newcommand{\K}{\alpha}

\begin{document}

\title[A Pl\"unnecke-Ruzsa inequality in compact abelian groups]{A Pl\"unnecke-Ruzsa inequality in compact abelian groups}

\author{Pablo Candela}
\address{Universidad Aut\'onoma de Madrid, and ICMAT\\ Ciudad Universitaria de Cantoblanco\\ Madrid 28049\\ Spain}
\email{pablo.candela@uam.es}

\author{Diego Gonz\'alez-S\'anchez}
\address{Universidad Aut\'onoma de Madrid, and ICMAT\\ Ciudad Universitaria de Cantoblanco\\ Madrid 28049\\ Spain}
\email{diego.gonzalezs@predoc.uam.es}

\author{Anne De Roton}
\address{Institut Elie Cartan de Lorraine, UMR 7502, Universit\'e de Lorraine, F-54506  Vandoeuvre-les Nancy, France}
\email{anne.de-roton@univ-lorraine.fr}
\subjclass[2010]{Primary 11B30; Secondary 28A05}

\begin{abstract}
The Pl\"unnecke-Ruzsa inequality is a fundamental tool to control the growth of finite subsets of abelian groups under repeated addition and subtraction. Other tools to handle sumsets have gained applicability by being extended to more general subsets of more general groups. This motivates extending the Pl\"unnecke-Ruzsa inequality, in particular to measurable subsets of compact abelian groups by replacing the cardinality with the Haar probability measure. This objective is related to the question of the stability of classes of Haar measurable sets under addition. In this direction the class of analytic sets is a natural one to work with. We prove a Pl\"unnecke-Ruzsa inequality for $K$-analytic sets in general compact (Hausdorff) abelian groups. We also discuss further extensions, some of which raise questions of independent interest in descriptive topology. \vspace{-0.3cm}
\end{abstract}
\vspace*{-0.3cm}
\maketitle
\section{Introduction}
\noindent The Pl\"unnecke-Ruzsa inequality is a central result in additive combinatorics, providing useful upper bounds for the cardinality of iterated sums and differences of a finite subset of an abelian group. The version of the result that is used most often is the following.
\begin{theorem}\label{thm:Plun-finite}
Let $A, B$ be finite non-empty subsets of an abelian group and suppose that $|A+B| \leq \K |A|$. Then for all non-negative integers $m,n$ we have $|mB-nB|\leq \K^{m+n} |A|$.
\end{theorem}
\noindent A first version of this result (for iterated sums only) was proved by Pl\"unnecke in the late 1960s \cite{Plun2}. The proof was simplified and the result extended to sums and differences by Ruzsa in the late 1980s \cite{Ruz1}. Both of these treatments of the result used nontrivial tools from graph theory. In 2011, a much shorter and elementary proof was given by Petridis \cite{Petridis1}. We refer the reader to the latter paper and also to the survey \cite{Petridis2} of the same author for more background on this result and its numerous applications.

Theorem \ref{thm:Plun-finite} is applicable in the discrete setting of \emph{finite} subsets of abelian groups. Other central tools to handle sumsets have gained much applicability by being extended from the discrete setting to more general settings including continuous groups. This is the case for instance for the Cauchy-Davenport inequality, which was extended to the circle group in \cite{Raik}, to tori in \cite{Macbeath}, and to compact connected abelian groups in \cite{Knes}. This motivates extending Theorem \ref{thm:Plun-finite} to more general subsets of more general abelian groups. Here we focus on Haar-measurable subsets of compact abelian groups, aiming for an extension of Theorem \ref{thm:Plun-finite} with the cardinality replaced by the Haar probability measure. This leads us to seek a suitable class of Haar-measurable sets for which to prove such an extension. Such a class should be sufficiently general, but it is also natural to require it to be stable under addition, meaning that if $A,B$ are sets in this class then so is their sumset $A+B$. Questions related to this stability were already of interest to Erd\H os and Stone, who showed in \cite{E&S} that the sum of two Borel sets can fail to be Borel. It is also 
known since Sierpi\'nski's work \cite{Sierp} that the sum of Lebesgue measurable sets need not be Lebesgue measurable (see also \cite{CFF}). However, the class of \emph{analytic sets} is stable under addition (as was already noted in \cite{E&S,Sierp}), and in a Polish space (a separable topological space metrizable by a complete metric) this class is general enough to contain all Borel sets; see Proposition 8.2.3 of \cite{cohn}. In this paper we extend the Pl\"unnecke-Ruzsa inequality to analytic sets in compact Polish abelian groups, and thus to all Borel sets in such groups. In fact, our main result holds for the more general class of \emph{$K$-analytic sets}, which can be defined in any compact (Hausdorff) abelian group, as we recall below.

There are also extensions of additive combinatorial tools to the non-abelian setting, for instance in \cite{Kemp} and more recently in \cite{Tao-prod}. The latter paper includes a variant of the Pl\"unnecke-Ruzsa inequality (with weaker bounds) for non-abelian groups (see Lemma 3.4 in \cite{Tao-prod}), and also related results for open sets in some continuous groups. The extensions in this paper go in a different direction, their aim being to make the Pl\"unnecke-Ruzsa inequality applicable to as large a class of sets as possible in the compact abelian setting.

Before we state our main result, let us recall some definitions. All compact abelian groups in the sequel are assumed to be Hausdorff. In the setting of general Hausdorff topological spaces, Choquet defined the useful notion of a \emph{$K$-analytic set}; see Definition 3.1 in \cite{choquet}. This extended the classical notion of analytic set defined by Lusin and Souslin \cite{Lus,Sous}, the latter notion pertaining to Polish spaces. To state Choquet's definition, let us first recall that a subset $B$ of a topological space $X$ is a \emph{$K_{\sigma \delta}$ set} if $B= \bigcap_{i\in \mb{N}}\bigcup_{j\in\mb{N}} K_{i,j}$, for compact sets $K_{i,j}\subset X$. 
\begin{defn}[Choquet]\label{def:Choquet-ana}
Let $X$ be a Hausdorff topological space. A set $A\subset X$ is a  \emph{$K$-analytic set} if there is a $K_{\sigma \delta}$ set $B$ in a compact Hausdorff space and a continuous map $f:B\to X$ such that $A=f(B)$.
\end{defn}
\noindent We recall more background on $K$-analytic sets in Section \ref{sec:Gen} below.  Let us note for now that the sum or difference of two $K$-analytic sets in a compact abelian group $G$ is $K$-analytic (this follows from the definition, and is detailed in Section \ref{sec:Gen}), and that $K$-analytic subsets of $G$ are Haar-measurable; see Theorem 4.3 in \cite{sion}. 

Let $\mu$ denote the Haar probability measure on $G$. We can now state our main result.
\begin{theorem}\label{thm:final-result}
Let $G$ be a compact abelian group and let $A,B$ be $K$-analytic subsets of $G$ satisfying $0<\mu(A+B) \leq \K\,\mu(A)$. Then we have $\mu(mB-nB) \leq \K^{m+n}\mu(A)$ for all non-negative integers $m,n$.
\end{theorem}
\noindent As mentioned above, if $G$ is also Polish then the theorem holds in particular for any Borel sets $A,B\subset G$. We also prove the following variant, which can be useful in cases where the constant $\K\geq 1$ is close to 1.
\begin{theorem}\label{thm:final-result2}
Let $G$ be a connected compact abelian group and let $A,B$ be $K$-analytic subsets of $G$ satisfying $0<\mu(A+B) \leq \K\,\mu(A)$. Then for all $n\in \mb{N}$ such that $\K^n<1/\mu(A)$, we have $\mu(nB) \leq (\K^n-1)\,\mu(A)$.
\end{theorem}
\noindent This result also holds in finite groups in which an analogue of the Cauchy-Davenport inequality is available, as we shall explain in the sequel.

The condition $0<\mu(A+B)$ in the theorems above is necessary. Indeed, let $C$ be the Cantor middle-third set in $[0,1]$, and let $B$ denote $C$ viewed as a subset of the circle group $\mb{T}=\mb{R}/\mb{Z}$ (identifying this group as a set with $[0,1)$ the usual way). Since in $\mb{R}$ we have $C+C=[0,2]$ (as can be seen using ternary expansions), in $\mb{T}$ we have $\mu(mB-nB)=1$ whenever $m+n\geq 2$. If we now let $A$ be a singleton in $\mb{T}$, then $\mu(A+B)=\mu(B)=0$. In particular, for every $\K>0$ we have $\mu(A+B) \leq \K\,\mu(A)$, but the conclusion of Theorem \ref{thm:final-result} fails for all $m,n\in\mb{N}$.

The paper is laid out as follows. In Section \ref{sec:closed} we establish the special case of Theorem \ref{thm:final-result} for closed subsets of an arbitrary compact abelian Lie group. In Section \ref{sec:approx}, we prove an approximation result for closed subsets of general compact abelian groups by subsets of compact abelian Lie groups, which refines a similar result from \cite{CSV}. This is then combined in Section \ref{sec:Gen} with measure-theoretic results concerning $K$-analytic subsets of Hausdorff spaces, and using this we complete the proofs of Theorems \ref{thm:final-result} and \ref{thm:final-result2}. In Section \ref{sec:finrems} we discuss further extensions of Theorem \ref{thm:final-result}. In particular we prove a version of Theorem \ref{thm:final-result} involving the inner Haar measure, which allows the set $A$ to be arbitrary; see Theorem \ref{thm:more-general-version}. We then discuss further possible extensions of Theorem \ref{thm:final-result} to more general classes of Haar measurable sets, a direction which leads to basic questions in descriptive topology concerning generalizations of $K$-analytic sets (see for instance Question \ref{q:class}).
\medskip

\noindent \textbf{Acknowledgements.} We are very grateful to Petr Holick\'y for providing the example in Proposition \ref{prop:Holi} and for very useful comments. We thank an anonymous referee for useful remarks. This work was supported by project ANR-12-BS01-0011 CAESAR and by grant MTM2014-56350-P of MINECO. The second named author is supported by La Caixa.

\section{The case of closed sets in compact abelian Lie groups}\label{sec:closed} 
\noindent Every compact abelian Lie group is isomorphic to $\mb{T}^d\times Z$ for some non-negative integer $d$ and some finite abelian group $Z$; see Proposition 2.42 in \cite{H&Mo}. In this section we prove the following special case of Theorem \ref{thm:final-result}.

\begin{theorem}\label{thm:plunnecke-closed} Let $A,B$ be closed subsets of $\mb{T}^d\times Z$ satisfying $0<\mu(A+B)\le \K\,\mu(A)$. Then for all non-negative integers $\ell,m$ we have $\mu(\ell B- mB)\leq \K^{\ell+m}\mu(A)$.
\end{theorem}

\begin{remark}
The sum or difference of any finite number of closed sets in a compact abelian group is closed. Indeed, in a compact Hausdorff space a set is closed if and only if it is compact. Therefore, the sum of any (finite) number of closed sets is the image of a compact set through a continuous map, so it is compact, whence it is also closed.
\end{remark}

Given a set $A\subset \mb{T}^d\times Z$, and a positive integer $n$, we define the set
\begin{equation}
A_n\;=\; A+\big([-\tfrac{1}{n},\tfrac{1}{n}]^d\times \{0_Z\}\big)\,\subset\, \mb{T}^d\times Z.
\end{equation}
\begin{remark}
The sequence of sets $(A_n)_{n\in \mb{N}}$ is decreasing and $\cap_{n\in\mb{N}}A_n=\overline{A}$. In particular, for a closed set $A$, by continuity of $\mu$ we have $\mu(A_n)\rightarrow\mu(A)$. 
\end{remark}

\noindent Theorem \ref{thm:Plun-finite} is usually deduced from the following result (see Theorem 3.1 in \cite{Petridis1}).

\begin{theorem}\label{thm:pr}
Let $A$ and $B$ be finite non-empty subsets of an abelian group satisfying $|A+B|\leq \K|A|$. Then there exists a non-empty subset $X\subset A$ such that for every positive integer $m$ we have $|X+ m B| \leq \K^m |X|$.
\end{theorem}

\noindent In the same spirit, we shall first establish the following analogue of Theorem \ref{thm:pr} for closed subsets of compact abelian Lie groups. 

\begin{theorem}\label{thm:pr-closed}
Let $A,B$ be closed subsets of $\mb{T}^d\times Z$ satisfying $0<\mu(A+B)\leq \K\, \mu(A)$. Then for every $\epsilon>0$, for every sufficiently large $n\in \mb{N}$ there exists a non-empty closed subset $A'_n\subset A_n$ such that for every $m\in\mb{N}$ we have $\mu(A'_n+m B)\leq (1+\epsilon)^m\, \K^m\, \mu(A'_n)$.
\end{theorem}
\noindent Let us record a consequence that we shall use later to obtain Theorem \ref{thm:final-result2}.
\begin{corollary}\label{cor:cauchy}
Let $A,B$ be closed subsets of $\mb{T}^d$ satisfying $0<\mu(A+B)\leq \K\, \mu(A)$. Then for every positive integer $m$ such that $\K^m<1/\mu(A)$, we have $\mu(mB)\leq (\K^m-1) \mu(A)$.
\end{corollary}
\noindent The condition $\K^m<1/\mu(A)$ is seen to be necessary by letting $A=B$ with $\mu(A)>1/2$, $m=1$, and $\K=\tfrac{1}{\mu(A)}<2$. We then have $\mu(A+B)=1=\K\, \mu(A)$, yet $\mu(B)>(\K-1)\mu(A)$.
\begin{proof}
Let $(\epsilon_j)_{j\in \mb{N}}$ be a decreasing sequence of positive numbers tending to 0, and let $(n_j)_{j\in \mb{N}}$ be a strictly increasing sequence of positive integers such that for every $j$ there exists a closed set $A'_{n_j}\subset A_{n_j}$ satisfying $\mu(A'_{n_j}+m B)\;\leq (1+\epsilon_j)^m \K^m \mu(A'_{n_j})$. From the assumption $\K^m< 1/\mu(A)$ it follows that $\mu(A'_{n_j}+m B)\leq (1+\epsilon_j)^m \K^m \mu(A'_{n_j})<1$ for $j$ sufficiently large, so we may apply Macbeath's analogue for $\mb{T}^d$ of the Cauchy-Davenport inequality, see Theorem 1 in \cite{Macbeath}, to deduce that $\mu(A'_{n_j}+m B)\geq \mu(A'_{n_j})+\mu(m B)$, whence $\mu(mB)\leq \big( (1+\epsilon_j)^m \K^m - 1\big)\, \mu(A'_{n_j}) \leq \big( (1+\epsilon_j)^m \K^m - 1\big)\, \mu(A_{n_j})$. Letting $j\to\infty$ and using the continuity of the Haar measure, the result follows.
\end{proof}

To prove Theorem \ref{thm:pr-closed} we begin with the following basic fact.

\begin{lemma}\label{lem:asym-closed-sum-approx}
Let $A,B\subset \mb{T}^d\times Z$ be closed sets. Then $\mu(A_n+B_n)\to \mu(A+B)$ as $n\to\infty$.
\end{lemma}
\begin{proof}
It suffices to prove that $\bigcap_{n\in \mb{N}} (A_n+B_n)=A+B$. Indeed, since for each $n$ we have $A_n+B_n=A+B+\big([-\tfrac{2}{n},\tfrac{2}{n}]^d\times \{0_Z\}\big)$, the sequence of sets $(A_n+B_n)_{n\in \mb{N}}$ is decreasing, so the result would then follow by continuity of $\mu$. Clearly $\bigcap_{n\in \mb{N}} (A_n+B_n) \supset A+B$. To see the opposite inclusion, let $x\in \bigcap_{n\in \mb{N}} (A_n+B_n)$. For every $n$ let $a_n\in A_n$, $b_n\in B_n$ such that $x=a_n+b_n$. There is a convergent subsequence $(a_k)$ of $(a_n)$ and, within the resulting set of integers $k$, there is an infinite subset of integers $\ell$ such that $(b_\ell)$ converges as well. We thus have $a,b\in \mb{T}^d\times Z$ such that $a_\ell\to a$ and $b_\ell\to b$ as $\ell\to\infty$, and $a_\ell+b_\ell=x$ for every $\ell$. Since $a_\ell\in A_\ell$ and $b_\ell\in B_\ell$, by definition of these sets there exist $a_\ell'\in A$, $b_\ell'\in B$ such that $a_\ell-a_\ell'$ and $b_\ell-b_\ell'$ both  converge to $0$ as $\ell\to\infty$. Hence $a_\ell'\to a$ and $b_\ell'\to b$ as $\ell\to\infty$. Since $A,B$ are closed, we have $a\in A$ and $b\in B$. Hence $x=\lim_{\ell\to\infty} (a_\ell+b_\ell)=\lim_{\ell\to\infty} a_\ell+\lim_{\ell\to\infty} b_\ell = a+b\in A+B$, as required.
\end{proof}

\begin{proof}[Proof of Theorem \ref{thm:pr-closed}]
Fix any $\epsilon>0$. Since $\mu(A+B)>0$, by Lemma \ref{lem:asym-closed-sum-approx} there exists $n$ such that $\mu(A_n+B_n)\leq (1+\epsilon)\,\mu(A+B)$.  
We also have $\mu(A)>0$ so, since $\cap_n A_n=A$, we can also suppose that $\mu(A_n)\leq (1+\epsilon)\mu(A)$, by taking $n$ even larger if necessary.

Consider now $A_{2n}$, $B_{2n}$, which also satisfy $\mu(A_{2n}+B_{2n})\leq (1+\epsilon)\mu(A+B)$.\\
Let $N=2n$, and consider the following discrete subgroup of $\mb{T}^d\times Z$ (where we identify $\mb{T}^d$ as a set with $[0,1)^d$, and $\mb{Z}_N$ denotes the integers in $[0,N-1]$ with addition mod $N$):
\[
(\tfrac{1}{N}\mb{Z}^d_N)\times Z:=\big\{(\tfrac{j}{N},z):j\in\{0,\dots,N-1\}^d, z\in Z\big\}.
\]
We denote the small cube $\big[0,\tfrac{1}{N}\big)^d\times \{0_Z\}$ by $Q$, and define the following subsets of $\mb{T}^d\times Z$:
\begin{align*}
D_A & :=  \Big\{ \big(\tfrac{j}{N},z\big)\, \in \, (\tfrac{1}{N}\mb{Z}^d_N)\times Z ~:~ \Big(\big(\tfrac{j}{N},z\big) + Q\Big)\cap A_{2n} \neq \emptyset \Big\},\\
D_B & := \Big\{ \big(\tfrac{j}{N},z\big)\, \in \, (\tfrac{1}{N}\mb{Z}^d_N)\times Z ~:~ \Big(\big(\tfrac{j}{N},z\big) + Q\Big)\cap B_{2n} \neq \emptyset \Big\}.
\end{align*}
We claim that
\begin{equation}\label{eq:A2n}
A_{2n} \;\;\subset\;\; D_A+ Q \;\;\subset\;\; A_n.
\end{equation}
To see the first inclusion, note that for every $x\in A_{2n}$ there exists a unique $j\in [0,N)^d$ and $z\in Z$ such that $x\in\, \big(\tfrac{j}{N},z\big) + Q$, and then by definition we have $\big(\tfrac{j}{N},z\big)\in D_A$.  
To see the second inclusion, note that for every $x\in D_A+ Q$ there is $\big(\tfrac{j}{N},z\big)\in D_A$ such that $\big(\tfrac{j}{N},z\big) + Q$ contains both $x$ (by assumption) and also some $a\in A_{2n}$ (by definition of $D_A$). Therefore $x\; \in \; a+\big( \big(-\tfrac{1}{N},\tfrac{1}{N}\big)^d\times \{0_Z\}\big) \; \subset \; A_{2n}+\big(\big(-\tfrac{1}{N},\tfrac{1}{N}\big)^d\times \{0_Z\}\big)$. Since this holds for every such $x$, it follows that $D_A+Q \subset  A_{2n}+\big(\big(-\tfrac{1}{N},\tfrac{1}{N}\big)^d\times \{0_Z\}\big) \subset A_n$.

In exactly the same way, we obtain that
\begin{equation}\label{eq:B2n}
B_{2n} \;\;\subset\;\; D_B+ Q\;\; \subset\;\; B_n.
\end{equation}

We now claim that
\begin{equation}\label{eq:asimdiscsum}
\big|\,D_A \; + \; D_B +\big(\big\{0,\tfrac{1}{N}\big\}^d\times \{0_Z\}\big)\,\big| \leq (1+\epsilon) \,\K\, |D_A|.
\end{equation}
Indeed, the left side equals $N^d\, |Z|\, \mu\Big(D_A + D_B +\big(\big\{0,\tfrac{1}{N}\big\}^d\times \{0_Z\}\big) + Q \Big)$, which is
\[
N^d\, |Z|\, \mu\Big(D_A + D_B +\big([0,\tfrac{2}{N})^d\times \{0_Z\}\big)\Big) \;
 = \; N^d\, |Z| \, \mu\Big( D_A+ Q \; + \;D_B+ Q \Big).
\]
By \eqref{eq:A2n} and \eqref{eq:B2n}, this is at most $N^d\, |Z| \, \mu(A_n + B_n)$. By our choice of $n$ and our assumptions, this is at most $N^d\, |Z|\, (1+\epsilon)\, \mu(A+B) \; \leq \; N^d\, |Z| \, (1+\epsilon)\, \K\, \mu(A) \; \leq  \; N^d\, |Z| \, (1+\epsilon) \, \K\, \mu(A_{2n})$. By \eqref{eq:A2n} this is at most $N^d\, |Z|\, (1+\epsilon) \,\K\, \mu\big(D_A+ Q\big) =   (1+\epsilon)\, \K \, |D_A|$, and \eqref{eq:asimdiscsum} follows.

Now, given \eqref{eq:asimdiscsum}, we apply Theorem \ref{thm:pr} to $D_A$ and $D_B +\big(\big\{0,\tfrac{1}{N}\big\}^d\times \{0_Z\}\big)$ in the finite group $\tfrac{1}{N}\mb{Z}^d_N\times Z$, and we obtain a set $D_{A'}=D_{A',n}\subset D_A$ such that for every $m\geq 1$
\[
\big|D_{A'}+m \big(D_B+\big(\{0,\tfrac{1}{N}\}^d\times \{0_Z\}\big)\,\big)\big| \leq (1+\epsilon)^m\, \K^m\, |D_{A'}|.
\]
Let $A'_n=D_{A'}+\overline{Q}=D_{A'}+\big([0,\tfrac{1}{N}]^d\times \{0_Z\}\big)$, which is a closed subset of $A_n$. Using \eqref{eq:B2n} we have 
$A'_n+m B_{2n}\; \subset \; D_{A'}+\overline{Q}+ m\big(D_B\,+\,Q\big)=D_{A'} + m D_B+ \big(\big[0,\tfrac{m+1}{N}\big)^d\times \{0_Z\}\big)$, and this last set in turn is $D_{A'} + m D_B \, + \, \big(\big\{0,\tfrac{1}{N},\dots, \tfrac{m}{N}\big\}^d\times \{0_Z\}\big) \, + \, Q$, which equals $D_{A'} \, + \, m \big( D_B \,+ \,\{0,\tfrac{1}{N} \}^d\times \{0_Z\}\big) \, + \, Q$. Note that this last set has measure equal to $N^{-d}\;|Z|^{-1}\; \big|\,D_{A'} + m (D_B \, + \, \{0,\tfrac{1}{N} \big\}^d\times \{0_Z\}\big)\,\big|$. 
Hence
\begin{align*}
\mu(A'_n+m B)& \leq  \mu(A'_n+m B_{2n}) \; \leq \;  N^{-d}\;|Z|^{-1}\; \big|\,D_{A'} + m (D_B \, + \, \{0,\tfrac{1}{N} \big\}^d\times \{0_Z\}\big)\,\big| \\
& \leq  (1+\epsilon)^m\, \K^m\,N^{-d}\,|Z|^{-1}\, |D_{A'}|  = (1+\epsilon)^m \K^m \mu(A'_n).
\hspace{2.8cm}\qedhere\end{align*}
\end{proof}

\noindent To deduce Theorem \ref{thm:plunnecke-closed}, we emulate the argument from the discrete setting, which uses Ruzsa's triangle inequality. To do so we use the following generalization of this inequality, which follows directly from the proof of a more general version  by Tao (valid also in the non-commutative setting), namely  Lemma 3.2 in \cite{Tao-prod}.

\begin{lemma}\label{lem:ruzsa-triangle-inequality} Let $A_1,A_2,A_3$ be closed subsets of a compact abelian group with Haar measure $\mu$. Then
$\mu(A_1-A_3)\; \mu(A_2) \leq \mu(A_1-A_2)\;\mu(A_2-A_3)$.
\end{lemma}
\noindent The main result of this section can now be obtained.
\begin{proof}[Proof of Theorem \ref{thm:plunnecke-closed}] We apply Theorem \ref{thm:pr-closed} to $A$ and $B$ with any fixed $\epsilon>0$, and obtain that for all $n$ sufficiently large $\mu(A_n'+m B) \le (1+\epsilon)^m\, \K^m\, \mu(A_n')$, for some $A_n'\subset A_n$ closed and any integer $m\geq 0$. If one of $\ell$ or $m$ is 0, say $\ell=0$, then we have immediately $\mu(m B)\leq \mu(A_n'+m B) \leq (1+\epsilon)^m \K^m\mu(A_n)$, and so letting $n\to\infty$, using that $\cap_{n\ge 1} A_n = A$, and then letting $\epsilon\to 0$, we deduce that $\mu(m B)\leq \K^m \mu(A)$ as required. If $\ell,m$ are both positive, then by Lemma \ref{lem:ruzsa-triangle-inequality} applied with $A_1=\ell B$, $A_2=-A_n'$, $A_3=mB$, we have
\begin{align*}
\mu(\ell B-mB)\; \mu(A_n') & \leq \mu(\ell B+A_n')\; \mu(A_n'+mB)\; \leq \;  (1+\epsilon)^{\ell+m}\; \K^{\ell+m} \; \mu(A_n')^2\\
& \leq  (1+\epsilon)^{\ell+m}\; \K^{\ell+m} \; \mu(A_n')\; \mu(A_n).
\end{align*}
From the proof of Theorem \ref{thm:pr-closed} we have $\mu(A_n')>0$. Dividing by this and letting $n\to \infty$, we obtain $\mu(\ell B-mB) \le (1+\epsilon)^{\ell+m}\; \K^{\ell+m}\; \mu(A)$. Letting $\epsilon\to 0$, the result follows.
\end{proof}
\noindent To complete the proof of Theorem \ref{thm:final-result}, firstly, in the next section we approximate any compact abelian group by a Lie group in such a way that Theorem \ref{thm:plunnecke-closed} can be used to deduce the case of Theorem \ref{thm:final-result} for \emph{closed} sets. Then in Section \ref{sec:Gen}, using approximation results for $K$-analytic sets in Hausdorff spaces, we deduce Theorem \ref{thm:final-result} in full generality.

\section{Extension to closed subsets of compact abelian groups}\label{sec:approx}
\noindent Approximating compact groups by compact Lie groups is a standard technique,  and it has been used already in arithmetic combinatorics (e.g. in \cite{CSV}). However, here we shall need such approximations with the added guarantee that they behave well with respect to addition. We ensure this by working with closed sets, obtaining the following result.
\begin{lemma}\label{lem:approx-general-cpct-gps}
Let $G$ be a compact abelian group with Haar probability measure $\mu$, let $A,B$ be closed subsets of $G$, and let $\delta>0$. Then there exists a  compact abelian Lie group $G_0$, a continuous surjective homomorphism $q:G\to G_0$, and closed sets $A',B' \subset G_0$, such that $A\subset q^{-1}(A')$, $B\subset q^{-1}(B')$, $\mu(q^{-1}(A')\setminus A)<\delta$, and $\mu\big(q^{-1}(A'+B')\setminus (A+B)\big)<\delta$.
\end{lemma}

\begin{remark}
The proof of this lemma will make it clear that we would be able to approximate simultaneously any finite number of sets, as well as combinations of them using sum and difference. For example, given closed sets $A_1,A_2,A_2$ we could obtain sets $A_i'$ in $G_0$ such that $A_i\subset q^{-1}(A_i')$, $\mu(q^{-1}(A_i')\setminus A_i)<\delta$ for $i=1,2,3$, and also $\mu\big(q^{-1}(A_1'+A_2')\setminus (A_1+A_2)\big)<\delta$ and $\mu\big(q^{-1}(A_1'+A_2'-2A_3')\setminus (A_1+A_2-2A_3)\big)<\delta$.
\end{remark}

\noindent To prove Lemma \ref{lem:approx-general-cpct-gps}, we first prove the following modification of Lemma A.2 in \cite{CSV}.

\begin{lemma}\label{lem:CSV} Let $G$ be a compact abelian group, let $A$ be a closed subset of $G$, and let $0<\delta<1/2$. Then there exists a compact abelian Lie group $G_0$ and a continuous surjective homomorphim $q:G\to G_0$ such that, letting $A'=q(A)$, we have $\mu(q^{-1}(A')\setminus A) < \delta$.
\end{lemma}

\begin{proof} By regularity of $\mu$, there is an open set $U\supset A$ such that $\mu(U\setminus A)<\delta^3/2^{10}$. By Urysohn's lemma (see Theorem 32.3 and Theorem 33.1 in \cite{munkres}) there is a continuous function $h:G \to [0,1]$ such that $h(x)=1$ for all $x\in A$ and $h(x)=0$ for all $x\notin U$. Then
\[
\|1_A-h\|_{L^{1}(G)} = \int_A|1_A-h|\,\ud\mu+\int_{U\setminus A}|1_A-h|\,\ud\mu + \int_{U^c} |1_A-h|\,\ud\mu.
\]
The first and last integrals are 0, and the second one is at most $\mu(U\setminus A)<\delta^3/2^{10}$.

By the Stone-Weierstrass theorem, there is a trigonometric polynomial $P(x)$ such that $\|h-P\|_{L^{\infty}(G)}<\delta^3/2^{10}$ (see p. 24 in \cite{rudin}), whence $\|P-1_A\|_{L^{1}(G)}<\delta^3/2^{9}$. By the triangle inequality we also have $\|P\|_{L^{\infty}(G)}<2$. Here the proof differs from that of Lemma A.2 in \cite{CSV}: here $|P(a)-1|<\delta^3/2^{10}$  holds \emph{for all} $a\in A$ (we use this at the end of the proof).

Let $\widehat{G}$ be the dual group of $G$ and let $\widehat{G_0}$ be the subgroup of $\widehat{G}$ generated by the spectrum of $P$, i.e., by the finite set $\{\gamma\in \widehat{G}:\widehat{P}(\gamma) \not=0\}$. Then $\widehat{G_0}$ is a finitely generated discrete abelian group, therefore it is the dual of a compact abelian Lie group $G_0$. Letting $\Lambda$ be the annihilator of $\widehat{G_0}$ ($\Lambda$ a closed subgroup of $G$), we have that $G_0$ is isomorphic as a compact abelian group to $G/\Lambda$ (see \cite{rudin} section 2.1), so the map $G\to G/\Lambda$ gives a continuous surjective homomorphism $q:G\to G_0$. Then there exists a trigonometric polynomial $P_0$ on $G_0$ with $P=P_0 \circ q$, whence $\|P_0\|_{L^{\infty}(G_0)}\le 2$. Moreover, writing $P-P^2=P-1_A+1_A^2-P^2$, we have
\begin{align*}
\left\|P_0-P_0^2\right\|_{L^1(G_0)} & = \left\|P-P^2\right\|_{L^1(G)} \\ 
& \le \int_G |1_A-P|\,\ud\mu_G+\int_G|1_A-P|\, |1_A+P|\, \ud\mu_G < \delta^3/2^7.
\end{align*}
Therefore, the set $D:=\{x\in G_0:|P_0(x)-P_0^2(x)|>\delta^2/2^4\}$ has measure at most $\delta/8$. For every $x$ in the complement $D^c = G_0\setminus D$, we must have $|P_0(x)|\le \delta/4$ or $|1-P_0(x)|\le \delta/4$.

Now let $A_0:=\{x\in G_0:|P_0(x)-1|\le \delta/4\}$. We have that $\|1_{A_0}-P_0\|_{L^1(G_0) }$ is at most
\[
3\int_{D} \ud\mu_{G_0}+\int_{A_0\cap D^c}|1-P_0(x)|\ud\mu_{G_0}+\int_{A_0^c\cap D^c}|P_0(x)|\ud\mu_{G_0}< 7\delta/8,
\]
so $\mu_G(A\bigtriangleup q^{-1}(A_0)) \le \|1_A-P\|_{L^1}+\|P-1_{A_0} \circ q\|_{L^1} = \|1_A-P\|_{L^1(G)}+\|1_{A_0}-P_0\|_{L^1(G_0)}<\delta $.

Now note that by definition of $A_0$ it is clear that $A':=q(A)$ is included in $A_0$, because $|1-P(a)|<\delta^3/2^{10}$ for all $a\in A$. So indeed, we have that $\mu_G( q^{-1}(A_0)\setminus A)<\delta$. Moreover, instead of taking $A_0$ as our approximating set, we can just take $A'$, since $A\subset q^{-1}(A') \subset q^{-1}(A_0)$.
\end{proof}

\begin{proof}[Proof of Lemma \ref{lem:approx-general-cpct-gps}] 
First, by the same argument as in the proof of Lemma \ref{lem:CSV} applied to   the sets $A$ and $A+B$, we find polynomials $P_1$ and $P_2$ that yield the  approximations for $A$ and $A+B$ respectively. Then, to define  $\widehat{G_0}$, instead of the spectrum of $P$ as in the previous proof, now we take $\widehat{G_0}$ to be the subgroup generated by the union of the spectra of $P_1$ and $P_2$, that is $\{\gamma\in \widehat{G}:\widehat{P_i}(\gamma) \not=0\textrm{ for $i=1$ or $2$} \}$. This is again a finite set, so $G_0$ is finitely generated as required. We then obtain the desired  approximation simultaneously for $A$ and $A+B$, namely that $\mu\big((q^{-1}q(A))\setminus A\big)$ and $\mu\big((q^{-1}q(A+B)) \setminus (A+B)\big)$ are both less than $\delta$. Then letting $A'=q(A)$ and $B'=q(B)$ and using that $q$ commutes with addition, the result follows.
\end{proof}
\noindent We can now obtain the claimed special case of Theorem \ref{thm:final-result}.

\begin{proof}[Proof of Theorem \ref{thm:final-result} for closed sets]
Let $A,B$ be closed sets in the compact abelian group $G$ such that $0<\mu(A+B)\leq \K\, \mu(A)$; in particular $\mu(A)>0$. Fix an arbitrary small $\delta > 0$, and apply Lemma \ref{lem:approx-general-cpct-gps} to obtain the corresponding approximating sets $A',B' \subset G_0$. Then we have $0<\mu(A+B)\leq \mu\big(q^{-1}(A')+q^{-1}(B')\big)= \mu\big(q^{-1}(A'+B')\big)<\mu(A+B)+\delta$, and $\mu(A+B)\leq \alpha\, \mu(A)\leq \alpha\, \mu(q^{-1}(A'))$. 
Letting $\mu_0$ denote the Haar  measure on $G_0$, by the basic fact that the continuous surjective homomorphism $q$ preserves the Haar measures (i.e. $\mu\circ q^{-1} = \mu_0$), we have $0<\mu_0(A'+B') \leq \Big(\K+\tfrac{\delta}{\mu(A)}\Big)\,\mu_0(A')$, where in the last inequality we used that $\mu_0(A')\geq\mu(A)$.

Applying Theorem \ref{thm:plunnecke-closed}, we obtain $\mu_0(mB'-nB') \leq \big(\K+\delta/\mu(A)\big)^{m+n}\mu_0(A')$, which implies that $\mu(q^{-1}(mB'-nB')) \leq \big(\K+\delta/\mu(A)\big)^{m+n}\mu(q^{-1}(A'))$. Since $mB-nB \subset q^{-1}(mB'-nB')$, by Lemma \ref{lem:approx-general-cpct-gps} we have $\mu(mB-nB) \le \big(\K+\delta/\mu(A)\big)^{m+n}\,\big(\mu(A)+\delta\big)$. Letting $\delta\to 0$, the result follows.
\end{proof}
A similar argument yields the following extension of Corollary \ref{cor:cauchy}.

\begin{corollary}\label{cor:cauchy2}
Let $A,B$ be closed subsets of a connected compact abelian group satisfying $0<\mu(A+B)\leq \K\,\mu(A)$. Then for every $m\in \mb{N}$ such that $\K^m<1/\mu(A)$, we have
\[
\mu(mB)\leq (\K^m-1) \mu(A).
\]
\end{corollary}

\begin{proof}
We take $\delta>0$ so small that $\delta<\mu(A)$ and $(\K+\delta/\mu(A))^m<1/(\mu(A)+\delta)$. We can then argue as in the last proof, using the additional fact that the group $G_0$, being here a \emph{connected} compact abelian Lie group (by continuity of $q$ and connectedness of $G$), must be a torus $\mb{T}^d$, so that we can apply Corollary \ref{cor:cauchy}. 
\end{proof}

\section{Extension to $K$-analytic sets}\label{sec:Gen}
\noindent The classical definition of analytic sets, originating in work of Lusin and Souslin from 1917 \cite{Lus,Sous}, essentially concerned the type of spaces now known as Polish spaces. Let us recall the classical definition in this setting (see \cite{cohn}):   a subset $A$ of a Polish space $X$ is said to be an \emph{analytic set} if there is a Polish space $Y$ and a continuous function $f:Y\to X$ such that $f(Y)=A$. 

In the more general setting of Hausdorff spaces, as we recalled in the introduction, Choquet gave a fruitful definition of analytic sets that extends the classical one. For convenience we recall this definition here, but in a  slightly different form due to Sion, see Definition 2.1 in \cite{sion}. In Hausdorff spaces, the definitions of Choquet and Sion are in fact equivalent, as was shown by Jayne in \cite{jayne}. Recall that a subset $B$ of a topological space $X$ is a $K_{\sigma \delta}$ set if we have $B= \bigcap_{i\in \mb{N}}\bigcup_{j\in\mb{N}} K_{i,j}$, for compact sets $K_{i,j}\subset X$.

\begin{defn}[$K$-analytic set]\label{def:Sion-ana}
Let $X$ be a Hausdorff space. A set $A\subset X$ is a \emph{$K$-analytic set} if there is a $K_{\sigma \delta}$ set $B$ in a Hausdorff space and a continuous map $f:B\to X$ such that $A=f(B)$.
\end{defn}
\noindent The reference \cite{Rog} provides a detailed introduction to analytic sets, including historical background on the evolution of this notion.

In this section we extend the main result of the previous section to all $K$-analytic sets in a compact abelian group, thus completing the proof of Theorem \ref{thm:final-result}. Before we do so, let us briefly illustrate some consequences of this extension. 

When $X$ is a Polish space, Definitions \ref{def:Choquet-ana} and \ref{def:Sion-ana} are equivalent to the classical definition of analytic sets. Indeed, it is a basic fact that analytic sets in Polish spaces are continuous images of the set $\mc{I}=(0,1)\setminus\mb{Q}$ (see Proposition 8.2.7 and Example 6, p. 255 in \cite{cohn}), and it is not hard to see that $\mc{I}$ is a $K_{\sigma \delta}$ set. As mentioned in the introduction, we also have that in a Polish space all Borel sets are analytic (see Proposition 8.2.3 in \cite{cohn}). The extension of the Pl\"unnecke-Ruzsa inequality that we obtain in this section thus applies in particular to all Borel sets in any Polish compact abelian group. The family of Polish compact abelian groups contains every metrizable compact abelian group, see Corollary D.40 in \cite{cohn} (this includes for instance the Lie groups from Section \ref{sec:closed}, but also more general groups, for example $\mb{T}^\mb{N}$).

We now turn to the proof of Theorem \ref{thm:final-result}. Note first that it follows in a straightforward way from the distributivity of Cartesian products across unions and intersections that the Cartesian product of two $K$-analytic sets is $K$-analytic. From this it then follows, by continuity of addition, that if $A,B$ are $K$-analytic sets in $G$ then so is $A+B$. 

To complete the proof of the theorem, we shall use the following measure-theoretic property of $K$-analytic sets, which is a small modification (and follows from the proof) of Theorem 4.2 in \cite{sion}. 

\begin{theorem}\label{thm:sion} Let $X$ be a Hausdorff space, let $A_0$ be a $K_{\sigma \delta}$ set in some Hausdorff space, let $f:A_0\rightarrow X$ be a continuous map, let  $A$ be the $K$-analytic set $f(A_0)$ in $X$, and let $\mu$ be an outer measure on $X$. Then for every $\delta >0$ there is a compact set $C\subset A_0$ such that $\mu(f(C)) > \mu(A)-\delta$.
\end{theorem}
\noindent The standard definition of an outer measure (or Carath\'eodory measure) can be recalled from the same paper; see Definition 4.1 in \cite{sion}. We shall use the fact that the Haar measure $\mu$ on a compact abelian group is a restriction of an outer measure (namely the outer Haar measure) to the Haar-measurable sets, and the fact that $K$-analytic subsets of $G$ are Haar measurable (which follows from Theorem 4.3 of \cite{sion}). With these facts we can prove the following lemma, which plays a key role in the proof of Theorem \ref{thm:final-result}.

\begin{lemma}\label{lem:keyana}
Let $G$ be a compact abelian group, and let $B$ be a $K$-analytic subset of $G$. Then for all non-negative integers $m,n$ we have
\begin{equation}\label{eq:keyana}
\mu(mB-nB)=\sup_{D\subset B,\; D\textup{ compact}} \mu(mD-nD).
\end{equation}
\end{lemma}

\begin{proof}
The left side of \eqref{eq:keyana} is at least the right side since, on one hand, by the Haar measure's inner regularity we have $\mu(mB-nB)=\sup_{C\subset mB-nB,\; C \textrm{ compact}} \mu(C)$, and on the other hand for every compact set $D\subset B$ the set $C=mD-nD$ is a compact subset of $mB-nB$.

To see that the left side of \eqref{eq:keyana} is at most the right side, note that since $B^{m+n}$ is $K$-analytic there is a $K_{\sigma\delta}$ set $T$ in some Hausdorff space and a continuous function $f:T\to G^{m+n}$ such that $B^{m+n}=f(T)$. Let $\pm$ denote the continuous function $G^{m+n}\to G$, $(x_1,\ldots,x_{m+n})\mapsto x_1+\cdots+x_m-x_{m+1}-\cdots-x_{m+n}$. Fix any $\delta>0$, and note that by Theorem \ref{thm:sion} there exists a compact set $C\subset T$ such that $\mu\big(\pm(f(T))\big)-\delta< \mu\big(\pm(f(C))\big)$.

Let $D$ be the compact set $\pi_1(f(C)) \cup \cdots \cup \pi_{m+n}(f(C))$, where $\pi_i:G^{m+n}\to G$ is the projection to the $i$-th component. Since $f(C)\subset B^{m+n}$, it is clear that $D\subset B$. Moreover, we also have
\[
\pm(f(C))\subset \pi_1(f(C))+\cdots+\pi_m(f(C))-\pi_{m+1}(f(C))-\cdots-\pi_{m+n}(f(C))\subset mD-nD.
\]
Hence $\mu(mB-nB)-\delta=\mu\big(\pm(f(T))\big)-\delta< \mu(mD-nD)$. Since $\delta$ was arbitrary, the desired inequality follows, and the proof is complete.
\end{proof}

With these ingredients, we can now obtain our main result.

\begin{proof}[Proof of Theorem \ref{thm:final-result}]
We assume that $A,B$ are $K$-analytic subsets of $G$ that satisfy $0<\mu(A+B)\le \K\,\mu(A)$, so in particular $\mu(A)>0$. Fix an arbitrary $\delta\in (0,\mu(A))$. 

By Theorem \ref{thm:sion} there exists a compact set $E\subset A$ such that $\mu(E)> \mu(A)-\delta >0$, and by Lemma \ref{lem:keyana} there exists a compact set $D\subset B$ such that $\mu(mD-nD)>\mu(mB-nB)-\delta$. We then have 
\[
0<\mu(E+D) \leq  \mu(A+B) \le  \K\,\mu(A)\; \le\; \K\,(\mu(E)+\delta)
 \leq \K\big(1+\tfrac{\delta}{\mu(A)-\delta}\big)\mu(E).
\]
For subsets of the compact Hausdorff space $G$ the closure property is equivalent to compactness, so we can apply to $D,E$ the case of Theorem \ref{thm:final-result} for closed sets, obtaining
\[
\mu(mB-nB)-\delta \leq \mu(mD-nD) \le \K^{m+n}\big(1+\tfrac{\delta}{\mu(A)-\delta}\big)^{m+n}\mu(E).
\]
Letting $\delta \to 0$ and using that $E\subset A$, we deduce that $\mu(mB-nB)\le \K^{m+n}\mu(A)$.
\end{proof}
\noindent Theorem \ref{thm:final-result2} can be obtained with a similar argument, replacing the use of Theorem \ref{thm:final-result} for closed sets by that of Corollary \ref{cor:cauchy2}. 

\section{On further extensions of the main result}\label{sec:finrems}

\noindent In this last section we discuss further generalizations of Theorem \ref{thm:final-result}. In particular, in Subsection \ref{subsec:altermeas} we prove a version of the theorem that allows the Haar measure of one of the two sets to be replaced by the inner Haar measure, thus allowing this set to be arbitrary. Then, in Subsection \ref{subsec:SKA} we stick to using only the Haar measure and we discuss the problem of extending Theorem \ref{thm:final-result} to more general families of Haar measurable sets.

\subsection{Generalizing Theorem \ref{thm:final-result} using extensions of the Haar measure}\label{subsec:altermeas}\hfill \\
Theorem \ref{thm:final-result} yields the following more general version easily.

\begin{theorem}\label{thm:more-general-version} Let $G$ be a compact abelian group, with Haar measure $\mu$ and inner Haar measure $\mu_*$. Let $A$ be any subset of $G$, let $B$ a $K$-analytic subset of $G$, and suppose that $0<\mu_*(A+B) \leq \K\,\mu_*(A)$. Then for all non-negative integers $m,n$ we have
\[
\mu(mB-nB) \leq \K^{m+n}\mu_*(A).
\]
\end{theorem}
\noindent To prove this, first we note that Theorem \ref{thm:final-result} can be rephrased as follows.
\begin{theorem}\label{thm:final-result-alter}
Let $G$ be a compact abelian group, let $A,B$ be $K$-analytic subsets of $G$, with $\mu(A)>0$ and $B\neq \emptyset$. Then we have 
$\mu(mB-nB)\,\mu(A)^{n+m-1} \leq \mu(A+B)^{n+m}$ for all non-negative integers $m,n$.
\end{theorem}
\noindent The assumption $\mu(A)>0$ and $B\neq\emptyset$ here is indeed  equivalent to $0<\mu(A+B)\leq \K\, \mu(A)$, where we can always take the optimal constant, i.e.  $\K= \frac{\mu(A+B)}{\mu(A)}$.

\begin{proof}[Proof of Theorem \ref{thm:more-general-version}]
The assumption $0<\mu_*(A+B) \leq \K\,\mu_*(A)$ implies that $\mu_*(A)>0$.  We then have $\frac{\mu_*(A+B)}{\mu_*(A)}\leq \K$, and so it suffices to prove that for all $m,n\geq 0$ we have
\begin{equation}\label{eq:innerineq}
\mu(mB-nB)\,\mu_*(A)^{m+n-1} \leq \mu_*(A+B)^{m+n}.
\end{equation}
Let $E,F$ be compact subsets of $A,B$ respectively, with $\mu(E)>0$ and $F\neq\emptyset$. Then, by Theorem \ref{thm:final-result-alter} we have $
\mu(mF-nF)\,\mu(E)^{n+m-1} \leq \mu(E+F)^{n+m}$. Taking the supremum  of both sides of this inequality over compact sets $E\subset A$ and $F\subset B$, we have
\[
\sup_{F\subset B, \; F\text{ compact}}\mu(mF-nF)\sup_{E\subset A, \; E \text{ compact}}\mu(E)^{m+n-1} \leq \sup_{\substack{E\subset A,\; E \text{ compact} \\ F\subset B, \; F \text{ compact}}}\mu(E+F)^{m+n}.
\]
As $E+F$ is a compact subset of $A+B$, the right side here is at most $\mu_*(A+B)^{m+n}$. We also have $\sup_{E\subset A, \; E \text{ compact}}\mu(E)^{m+n-1}=\mu_*(A)^{m+n-1}$. Hence
\[
\sup_{F\subset B,\; F \text{ compact}}\mu(mF-nF)\,\mu_*(A)^{m+n-1} \leq \mu_*(A+B)^{m+n}.
\]
Since $B$ is $K$-analytic, applying \eqref{eq:keyana} we have 
$\sup_{F\subset B,\; F \text{ compact}}\mu(mF-nF) = \mu(mB-nB)$. This proves \eqref{eq:innerineq}, and the result follows.
\end{proof}
\noindent We do not know whether an even more general version of Theorem \ref{thm:more-general-version} holds in which both sets $A,B$ can be arbitrary. One difficulty is that to complete the above proof we relied on the property of  $K$-analytic sets given in \eqref{eq:keyana}, and we are not able to use such a property for more general sets. More precisely, to prove a more general version of Theorem \ref{thm:more-general-version} in which $B$ could also be arbitrary, it would be helpful to have an analogue of equality \eqref{eq:keyana} of the following kind holding for \textit{any} subset $B\subset G$: 
\begin{equation}\label{eq:difficult-step-2}
\sup_{F\subset B,\; F \text{ compact}}\mu(mF-nF) = \mu_*(mB-nB).
\end{equation}
However, this equality can fail. Indeed, we shall discuss a counterexample below that can be constructed using Bernstein sets. Bernstein sets in $\mb{R}$ are classical examples of non-measurable sets. Let us recall the definition of these sets in a Polish space.
\begin{defn}
A subset $B$ of a Polish space $X$ is a \emph{Bernstein set} if for every uncountable closed set $C \subset X$ we have $C\cap B \neq \emptyset$ and $C\setminus B \neq \emptyset$.
\end{defn}
\noindent Recall that a subset of a topological space is \emph{perfect} if it is closed and contains no isolated point. An equivalent definition of Bernstein sets in a Polish space $X$ is that $B$ is a Bernstein set in $X$ if it meets every nonempty perfect subset of $X$ but contains none of them (the equivalence can be seen using Corollary 6.3 and Theorem 6.4 of \cite{Kec}).

\begin{proposition}
There exists a set $B\subset \mb{T}$ for which equality \eqref{eq:difficult-step-2} fails for all $m,n\in \mb{N}$.
\end{proposition}
\begin{proof}
We can take $B$ to be a certain Bernstein subset of $\mb{T}$ that can be found using methods from \cite{kysiak}. The paper \cite{kysiak} provides constructions of Bernstein subsets of $\mb{R}$ with additional algebraic properties. Using Method 3.2 and Application 3.3 from \cite{kysiak}, we can construct a Bernstein subset $B\subset \mb{R}$ such that $B-B=\mb{R}$, and such that $B+\mb{Z}=B$ (for the latter property, which is not included in \cite{kysiak} explicitly, we can first ensure that $1\in B$, and since $B$ is a subgroup we obtain the desired property; we omit the details). Thus, we obtain a Bernstein set $B\subset \mb{T}$ with the property that $B-B = \mb{T}$. For this set we then have that  \eqref{eq:difficult-step-2} fails for all $m,n\in \mb{N}$. Indeed, we have on one hand $\mu_*(mB-nB)= \sup_{F\subset mB-nB,\; F \text{ compact}}\mu(F) = 1$, since $mB-nB\supset B-B = \mb{T}$, yet on the other hand $\sup_{F\subset B,\;F \text{ compact}}\mu(mF-nF) = 0$, since any such $F\subset B$ must be countable, so that $mF-nF$ is also countable and hence $\mu(mF-nF)=0$.
\end{proof}
\noindent Using Bernstein sets we can actually rule out at least one candidate of a version of Theorem \ref{thm:final-result} for arbitrary sets $A,B$, namely the version with assumption $0<\mu^*(A+B)\leq \alpha\mu^*(A)$ and conclusion $\mu_*(mB-nB) \leq \alpha^{m+n}\mu_*(A)$ for all $m,n \in \mb{Z}_{\geq 0}$. Indeed, we have the following example.
\begin{proposition} There exist subsets $A,B\subset \mb{T}$ such that $A+B$ is Haar measurable and satisfies $0<\mu(A+B)\leq \alpha\mu^*(A)$, and yet for all positive integers $m,n$ the set $mB-nB$ is Haar measurable and satisfies $\mu(mB-nB) > \alpha^{m+n}\mu_*(A)$.
\end{proposition}
\begin{proof}
Let $B\subset \mb{T}$ be the Bernstein set that we constructed above, satisfying $B-B=\mb{T}$, let $I=(-\tfrac{\epsilon}{2},\tfrac{\epsilon}{2}) \subset [-\tfrac{1}{2},\tfrac{1}{2})=\mb{T}$ for some $\epsilon < 1$, and let $A=B\cup I$. Since $B$ is dense in $\mb{T}$, we have $A+B \supset I+B = \mb{T}$, hence $\mu(A+B)=1$. We also have $\mu^*(A)\geq \mu^*(B)=1$, so $\mu(A+B) \leq \alpha\mu^*(A)$ with $\alpha=1$. However, for $m,n\ge 1$ we have $mB-nB \supset B-B = \mb{T}$, so $\mu(mB-nB) = 1$, and since $\mu_*(A) = \epsilon < 1$, we have $\mu(mB-nB) > \alpha^{n+m}\mu_*(A)$.
\end{proof}

\subsection{On extending Theorem \ref{thm:final-result} to larger families of Haar measurable sets}\label{subsec:SKA}\hfill \smallskip \\
As mentioned in the introduction, for non-Polish compact abelian groups one could desire a more general version of Theorem \ref{thm:final-result}, in particular because of the issue that the family of $K$-analytic sets does not necessarily contain all Borel subsets of such a group (see \S 5 in \cite{Hansell}). There are more recent, more general notions of analytic sets that do include all Borel sets in this setting. A notable example is the family of \v{C}ech-analytic sets.
\begin{defn}
Let $X$ be a compact Hausdorff topological space. A set $A$ in $X$ is a \emph{\v{C}ech-analytic set} if $A$ is the projection on $X$ of a set in $X\times \mb{N}^{\mb{N}}$ that is the intersection of a closed 
set with a $G_\delta$-set.
\end{defn}
\noindent This notion was introduced by Fremlin in the unpublished note \cite{Frem} (see also the appendix in \cite{JNR}). The family of \v{C}ech-analytic subsets of a compact Hausdorff space contains all Borel subsets of this space, as shown in Theorem 4 (c) of \cite{Frem}.

There is an even more general notion, namely that of a \emph{scattered-$K$-analytic set}. We shall mention this notion again below but we shall not recall its much more technical definition here (for more information on this notion we refer to \cite{Hansell,Holi93,Holi10}).

The families of \v{C}ech-analytic and scattered-$K$-analytic sets address several shortcomings, while conserving several main advantages, of the family of $K$-analytic sets in descriptive topology; this is discussed in Sections 5 and 6 of \cite{Rog}. It may therefore seem natural to wonder whether Theorem \ref{thm:final-result} holds for these families of sets. However, there is a property of $K$-analytic sets that fails for these more general families, namely the stability under addition in a compact abelian group. Because of this failure, we were unable to adapt the methods in this paper to extend Theorem \ref{thm:final-result} to these families. 

The main aim of this subsection is to illustrate this failure of stability under addition with an example, which was shown to us by Petr Holick\'y, and which we present below with his kind permission.
 
Recall that a subset $A$ of a topological space is \emph{isolated} if $A$ together with its relative topology is a discrete space (equivalently, the set contains no limit-point of itself). 
\begin{proposition}[P. Holick\'y]\label{prop:Holi}
There exists a compact abelian group $G$ and \v{C}ech-analytic sets $A,B\subset G$ such that $A+B$ is not \v{C}ech-analytic.
\end{proposition}

\begin{proof}
Let $G =\mb{T}^{\mb{R}}= \{f:\mb{R}\to \mb{T}\}$ equipped with pointwise addition. The pointwise topology on $G$ is compact (it is equivalent to the product topology and  compactness follows from Tychonoff's theorem). Thus $G$ is a compact abelian group.

Let $\{x_r:r \in  \mb{R} \setminus \{0\} \}$ be a non-analytic set in the Polish compact space $\mb{T}$. Viewing $\mb{T}$ as $[-\tfrac{1}{2},\tfrac{1}{2})$ with addition mod 1, we define
\begin{align}
A & =  \big\{f_r\in G: \;r\in \mb{R} \setminus \{0\},\; f_r(r)=\tfrac{1}{4},\; f_r(s)=0\textrm{ for }s\in \mb{R} \setminus \{r\}\big\},\\
B & =  \big\{g_r\in G:\; r\in \mb{R} \setminus \{0\},\; g_r(0)=x_r,\; g_r(r)=-\tfrac{1}{4},\;  g_r(s)=0\textrm{ for }s\in \mb{R} \setminus \{0,r\}\big\}. \nonumber
\end{align}
In these definitions we are taking for each real number $r\in \mb{R}\setminus \{0\}$ an element of $G$. For example, for $A$, the function $f_r(s)\in G$ takes the value $1/4$ when $s=r$ and $0$ otherwise.

We claim that $A$ and $B$ are both \v{C}ech-analytic subsets of $G$. To see this, recall  that the compact Hausdorff space $G$ is completely regular, and that complete regularity is a hereditary property. It then follows from Theorem 6.14 (c) of \cite{Hansell} that a  subset of $G$ is \v{C}ech-analytic if it is \emph{isolated-$K$-analytic} (in the sense of Definition 6.7 in \cite{Hansell}), so it suffices to show that $A$ and $B$ are both isolated-$K$-analytic. It can be seen that $A$ is isolated-$K$-analytic by noting that $A=\bigcup_{r\in \mb{R} \setminus \{0\}} \{f_r\}$, that each singleton $\{f_r\}$ is isolated, and that $\mathcal{E}=\{\{f_r\}:r\in \mb{R} \setminus \{0\}\}$ is an isolated collection (in the sense of Definition 6.1 of \cite{Hansell}); hence, by Theorem 6.13 (b) of \cite{Hansell}, the union $A$ is indeed isolated-$K$-analytic. Similarly $B$ is isolated-$K$-analytic, and our claim is thus proved.

For $\theta\in \mb{T}$ let $\|\theta\|_{\mb{T}}$ denote the absolute value of the representative of $\theta$ in $[-\tfrac{1}{2},\tfrac{1}{2})$. Let
\[
U=\big\{h\in G: \|h(r)\|_{\mb{T}}\leq \tfrac{1}{8}\textrm{ for }r\neq 0\big\}.
\]
This is a compact subset of $G$ so it is \v{C}ech-analytic (as a Borel set; see Theorem 4 in \cite{Frem}).

The family of \v{C}ech-analytic sets is closed under finite intersections (even countable ones, see Theorem 5.6 in \cite{Hansell}), so if $A+B$ were \v{C}ech-analytic, then $(A+B)\cap U$ would also be. However, we have $(A+B)\cap U=\{h\in G:  h(0)=x_r\textrm{ for some }r\neq 0,\; h(s)=0\textrm{ for }s\neq 0\}$. Hence $(A+B)\cap U$ is homeomorphic to a subset of $\mb{T}$ that is not analytic and therefore not \v{C}ech-analytic (in the Polish space $\mb{T}$ the classes of analytic and \v{C}ech-analytic sets are equal). Hence $(A+B)\cap U$ is not \v{C}ech-analytic, and then neither is $A+B$.
\end{proof}

\noindent Proposition \ref{prop:Holi} can be strengthened if we use the terminology  of scattered-$K$-analytic sets. Indeed, the family of scattered-$K$-analytic sets is larger than the family of  \v{C}ech-analytic sets and it can be shown that the set $A+B$ in Proposition \ref{prop:Holi} is not even scattered-$K$-analytic. The proof is actually almost the same, except that it uses the more technical definitions and properties of scattered-$K$-analytic sets given in \cite{Hansell,Holi93,Holi10}.

The extension of Theorem \ref{thm:final-result} mentioned at the beginning of this section, and Proposition \ref{prop:Holi}, together lead to the question of what would be a suitable class of Haar measurable sets, larger than the class of $K$-analytic sets, for which such an extension can be proved. 
\begin{question}\label{q:class}
Is there a class $\mc{C}$ of Haar measurable subsets of a general compact abelian group $G$ such that $\mc{C}$ is stable under addition and $\mc{C}$ contains every Borel subset of $G$?
\end{question}
\noindent If a generalization of Theorem \ref{thm:final-result} going beyond Theorem \ref{thm:more-general-version} is proved, in which $A$ and $B$ can both be arbitrary, then naturally Question \ref{q:class} will be less relevant to the Pl\"unnecke-Ruzsa inequality. However the question seems of interest in itself.

\end{document}